\documentclass[3p]{elsarticle}
\usepackage[hidelinks]{hyperref}
\usepackage{booktabs,multirow,array,caption}
\usepackage{tabularx}
\usepackage{graphicx}
\usepackage{float}
\usepackage{amsmath}
\usepackage{amssymb}
\usepackage{subfigure}
\usepackage{bigints}
\newcommand\eatpunct[1]{}
\newdefinition{proposition}{Proposition}[section]
\newdefinition{remark}{Remark}[section]
\newdefinition{definition}{Definition}[section]
\newdefinition{example}{Example}
\newdefinition{theorem}{Theorem}[section]
\newdefinition{lemma}{Lemma}[section]
\newdefinition{corollary}{Corollary}[section]
\newproof{proof}{\textbf{Proof}}
\allowdisplaybreaks
\begin{document}
\begin{frontmatter}
\title{Polynomial degree reduction in the $\mathcal{L}^2$-norm on a symmetric interval for the canonical basis}
\author{Habib Ben Abdallah\corref{cor}}
\ead{benabdallah-h@webmail.uwinnipeg.ca}
\author{Christopher J. Henry}
\ead{ch.henry@uwinnipeg.ca}
\author{Sheela Ramanna}
\ead{s.ramanna@uwinnipeg.ca}
\address{Department of Applied Computer Science, University of Winnipeg, Winnipeg, Manitoba, Canada}
\cortext[cor]{Corresponding author.}
\begin{abstract}
In this paper, we develop a direct formula for determining the coefficients in the canonical basis of the best polynomial of degree $M$ that approximates a polynomial of degree $N>M$ on a symmetric interval for the $\mathcal{L}^2$-norm. We also formally prove that using the formula is more computationally efficient than using a classical matrix multiplication approach and we provide an example to illustrate that it is more numerically stable than the classical approach.
\end{abstract}
\end{frontmatter}

\section{Introduction}\label{introduction}
Polynomials expressed in the canonical basis are widely used in different scientific fields such as machine learning, statistics and computer science. They can be used in different regression and classification models in machine learning \cite{machine_learning_1} \cite{machine_learning_2}, they can be used as the main building block for some statistical models \cite{statistical}, and they can also be used to compute hyperbolic and circular functions \cite{computer_science}. Although modern technology is able to rapidly exploit and manipulate polynomials, it is undeniable that saving computational and spatial resources is paramount for hardware and software optimization. Furthermore, in many cases, high degree polynomials are used on inputs (data) that are bounded in a certain interval. For instance, audio signal values are, in most of the cases, comprised between $-1$ and $1$, and grayscale images contain values that are comprised between $0$ and $255$ or $0$ and $1$. Moreover, some machine learning applications require polynomial-based models such as the 1-dimensional polynomial neural network model \cite{1DPNN} or the polynomial activation neural network model \cite{polynomial_activation} that use high degree polynomials in every neuron. Therefore, reducing the degree of a polynomial to obtain nearly the same results on a particular interval with increased spatial and computational efficiency can be vital for some new machine learning models to thrive.
\newline\indent
Polynomial reduction has been extensively explored for various norms to improve computer aided design (CAD) \cite{degree_reduction_1} \cite{degree_reduction_2} \cite{degree_reduction_3}, but the polynomials that are reduced are mostly expressed using Bernstein-Bézier coefficients because they are suitable to model graphical curves \cite{graphics}. However, the polynomials that are used in many machine learning models are usually expressed using coefficients in the canonical basis, or canonical coefficients. In fact, they can either be used as kernels \cite{machine_learning_2}\cite{kernels}\cite{intro_kernel_1}\cite{knn}\cite{intro_kernel_net}\cite{intro_kernel_net_1} or as core building blocks \cite{machine_learning_1}\cite{1DPNN}\cite{polynomial_activation}\cite{building_block_1}\cite{building_block_2}. The contribution of this work lies in the development of a direct formula to produce the canonical coefficients of a polynomial of low degree that approximates a polynomial of high degree which is expressed in the canonical basis. The reduced polynomial minimizes the $\mathcal{L}^2$-norm on any symmetric interval of the form $[-l,l]$ where $l\in\mathbb{R}_{+}^*$. We also demonstrated the theoretical proof that using the formula is more computationally efficient than using a classical approach consisting of matrix multiplications.
\newline\indent
The outline of this paper is as follows. In Section \ref{section:problem statement}, we define the problem and we provide a method to determine its solution. In Section \ref{section:orthonormal basis}, we define an orthonormal basis that is used as an intermediary for computing the canonical coefficients. In Section \ref{section:canonical basis}, we detail the steps that lead to determining the canonical coefficients of the reduced polynomial. In Section \ref{section:computational}, we analyze the computational complexities of the use of the direct formula and the use of the classical approach. We also provide two examples of how to use the results developed in this paper.
\section{Problem statement}\label{section:problem statement}
In this section, we define the problem and we provide a general way to determine the solution. Let $\mathbb{R}[X]$ be the polynomial algebra in one indeterminate $X$ over $\mathbb{R}$. Let $P,Q\in\mathbb{R}[X]$ such that $deg(Q)<deg(P)$. Let $N=deg(P)$ and $M=deg(Q)$. Let $(a_0,...,a_N)\in\mathbb{R}^{N+1}$ be the coefficients of $P$ in the canonical basis $(1,X,...,X^N)$ and $(b_0,...,b_M)\in\mathbb{R}^{M+1}$ be the coefficients of $Q$ in the canonical basis. Let $l\in\mathbb{R_+^*}$. We want to estimate the polynomial $Q$ that best approximates $P$ as such:
\begin{flalign*}
\min_{(b_0,...,b_M)}\dfrac{1}{2l}\int_{-l}^l\left(Q(x)-P(x)\right)^2dx=\min_{(b_0,...,b_M)}J_P(b_0,...,b_M).&&
\end{flalign*}
Determining the coefficients $(b_0,...,b_M)$ can be achieved by solving the linear equation $\nabla J_P=0$. However, solving the equation for these coefficients may prove to be difficult. Therefore, as an intermediate step, we suppose that there exists an orthonormal basis $(e_{0,l},...,e_{N,l})$ such that
\begin{flalign*}
\exists(\beta_0,...,\beta_M)\in\mathbb{R}^{M+1},Q=\sum_{n=0}^{M}\beta_ne_{n,l},&&
\end{flalign*}
and
\begin{flalign*}
\forall m,n\in[\![0,N]\!],\dfrac{1}{2l}\int_{-l}^{l}e_{m,l}(x)e_{n,l}(x)dx=\delta_{mn},&&
\end{flalign*}
where $\delta$ is the Kronecker delta. The objective is then to minimize $J'_P(\beta_0,...,\beta_M)=\dfrac{1}{2l}\int_{-l}^{l}(Q(x)-P(x))^2dx$ such that $Q$ is expressed in the orthonormal basis. Solving the equation $\nabla J'_P=0$ is equivalent to solving the following linear system:
\begin{flalign*}
\begin{split}
\forall m\in[\![0,M]\!],\dfrac{\partial J'_P}{\partial \beta_m}(\beta_0,...,\beta_M)=0&\iff\dfrac{1}{2l}\int_{-l}^{l}e_{m,l}(x)\left(\sum_{n=0}^M\beta_ne_{n,l}(x)-P(x)\right)dx=0\\
&\iff\sum_{n=0}^M\beta_n\dfrac{1}{2l}\int_{-l}^{l}e_{m,l}(x)e_{n,l}(x)dx=\dfrac{1}{2l}\int_{-l}^{l}P(x)e_{m,l}(x)dx\\
&\iff\beta_m=\dfrac{1}{2l}\int_{-l}^{l}P(x)e_{m,l}(x)dx=\sum_{n=0}^{N}a_n\dfrac{1}{2l}\int_{-l}^{l}x^ne_{m,l}(x)dx.
\end{split}&&
\end{flalign*} 
The solution of the system is then 
\begin{flalign}\label{eq:TMN}
\begin{pmatrix}
\beta_0\\
\vdots\\
\beta_M\\
\end{pmatrix}
=
T^{[M,N,l]}
\begin{pmatrix}
a_0\\
\vdots\\
\vdots\\
a_N
\end{pmatrix}
,&&
\end{flalign}
where $T^{[M,N,l]}$ is a $(M+1)\times (N+1)$ matrix such that
\begin{flalign*}
\forall (m,n)\in[\![0,M]\!]\times[\![0,N]\!],T^{[M,N,l]}_{m,n}= \dfrac{1}{2l}\int_{-l}^{l}x^ne_{m,l}(x)dx.&&
\end{flalign*}
$T^{[M,N,l]}$ is the orthonormal projection matrix of $\mathbb{R}_N[X]$ on $\mathbb{R}_M[X]$ from the canonical basis to the orthonormal basis. In order to determine the coefficients $(b_0,...,b_M)$ in the canonical basis, we use the transition matrix from the orthogonal basis to the canonical basis which happens to be $\left(T^{[M,M,l]}\right)^{-1}$ as such:
\begin{flalign*}
\left(T^{[M,M,l]}\right)^{-1}
\begin{pmatrix}
\beta_0\\
\vdots\\
\beta_M\\
\end{pmatrix}
=
\begin{pmatrix}
b_0\\
\vdots\\
b_M\\
\end{pmatrix}
=
\left(T^{[M,M,l]}\right)^{-1}T^{[M,N,l]}
\begin{pmatrix}
a_0\\
\vdots\\
\vdots\\
a_N
\end{pmatrix}
=
V^{[M,N,l]}
\begin{pmatrix}
a_0\\
\vdots\\
\vdots\\
a_N
\end{pmatrix}
.&&
\end{flalign*}
The objective is to determine $(b_0,...,b_M)$ by finding the elements of $V^{[M,N,l]}=\left(T^{[M,M,l]}\right)^{-1}T^{[M,N,l]}$ without the need to calculate the matrix product $\left(T^{[M,M,l]}\right)^{-1}T^{[M,N,l]}$.
\section{Determining the orthonormal projection of $\mathbb{R}_N[X]$ on $\mathbb{R}_M[X]$ in the orthonormal basis}\label{section:orthonormal basis}
In this section, we will define all the necessary notions that are needed to perform the orthonormal projection of $\mathbb{R}_N[X]$ on $\mathbb{R}_M[X]$. The notion of orthonormality is defined with respect to a scalar product. Thus, we first need to determine a suitable scalar product for the projection. We first define an intermediate scalar product that will serve as a basis for the remaining.
\begin{proposition}\label{proposition:scalar_product}
Let $\langle\cdot,\cdot\rangle$ be defined as such:
\begin{flalign*}
\begin{split}
\langle\cdot,\cdot\rangle\colon\mathbb{R}[X]\times\mathbb{R}[X]&\to\mathbb{R}\\
(A,B)&\mapsto\int_{-1}^{1}A(x)B(x)dx.
\end{split}&&
\end{flalign*}
$\langle\cdot,\cdot\rangle$ is a scalar product on $\mathbb{R}[X]$.
\end{proposition}
\begin{proof}
$\langle\cdot,\cdot\rangle$ is bilinear and symmetric since the integral is bilinear and symmetric. $\langle\cdot,\cdot\rangle$ is positive since $\forall A\in\mathbb{R[X]},\langle A,A\rangle=\int_{-1}^{1}A^2(x)dx\geq0$. If $\langle A,A\rangle=0$ then $A(x)=0,\forall x\in[-1,1].$ This means that $A$ has an infinite number of roots, therefore, $A=0_{\mathbb{R}[X]}$. Thus, $\langle\cdot,\cdot\rangle$ is definite and is, in turn, a scalar product on $\mathbb{R}[X]$.\indent$\blacksquare$
\end{proof}
\begin{remark}
This is a well-known result but we are formally proving it for the sake of clarity.
\end{remark}
We now define the scalar product that is relevant to our problem.
\begin{proposition}\label{test}
Let $l\in\mathbb{R_+^*}$ and $\langle\cdot,\cdot\rangle_l$ be defined as such:
\begin{flalign*}
\begin{split}
\langle\cdot,\cdot\rangle_l\colon\mathbb{R}[X]\times\mathbb{R}[X]&\to\mathbb{R}\\
(A,B)&\mapsto\dfrac{1}{2l}\int_{-l}^{l}A(x)B(x)dx.
\end{split}&&
\end{flalign*}
$\langle\cdot,\cdot\rangle_l$ is a scalar product on $\mathbb{R}[X]$.
\end{proposition}
\begin{proof}
The proof can be derived from the proof of Proposition \ref{proposition:scalar_product}.\indent$\blacksquare$
\end{proof}
The scalar product defined in Proposition \ref{proposition:scalar_product} is a well-known one that is related to a set of orthogonal polynomials called the Legendre polynomials.
\begin{definition}
The set of polynomials $(L_m)_{m\in\mathbb{N}}$ called Legendre polynomials is defined using Rodrigues formula \cite{rodrigues} as such:
\begin{flalign*}
\forall m\in\mathbb{N}, \forall x\in\mathbb{R}, L_m(x)=\dfrac{1}{2^mm!}\dfrac{d^m}{dx^m}\left[\left(x^2-1\right)^m\right].&&
\end{flalign*}
\end{definition}
\begin{lemma}
Legendre polynomials are orthogonal for $\langle\cdot,\cdot\rangle$.
\end{lemma}
\begin{proof}
Let $m,n\in\mathbb{N}$ such that $m\neq n$. We suppose that $m<n$. We have
\begin{flalign*}
\langle L_m, L_n\rangle=\int_{-1}^{1}L_m(x)L_n(x)dx=\dfrac{1}{2^{m+n}m!n!}\int_{-1}^{1}\dfrac{d^m}{dx^m}\left[\left(x^2-1\right)^m\right]\dfrac{d^n}{dx^n}\left[\left(x^2-1\right)^n\right]dx=\dfrac{1}{2^{m+n}m!n!}I_{m,n},&&
\end{flalign*}
where
\begin{flalign*}
I_{m,n} = \int_{-1}^{1}\dfrac{d^m}{dx^m}\left[\left(x^2-1\right)^m\right]\dfrac{d^n}{dx^n}\left[\left(x^2-1\right)^n\right]dx.&&
\end{flalign*}
By using integration by parts, we obtain
\begin{flalign*}
\begin{split}
I_{m,n}&=\left[\dfrac{d^m}{dx^m}\left[\left(x^2-1\right)^m\right]\dfrac{d^{n-1}}{dx^{n-1}}\left[\left(x^2-1\right)^n\right]\right]^1_{-1}-\int_{-1}^{1}\dfrac{d^{m+1}}{dx^{m+1}}\left[\left(x^2-1\right)^m\right]\dfrac{d^{n-1}}{dx^{n-1}}\left[\left(x^2-1\right)^n\right]dx\\
&=-I_{m+1,n-1}=(-1)^mI_{2m, n-m},
\end{split}&&
\end{flalign*}
because $\dfrac{d^{n-k}}{dx^{n-k}}\left[\left(x^2-1\right)^n\right]=0,\forall x\in\{-1,1\},\forall k\in[\![1,n]\!].$ This is due to the fact that $-1$ and $1$ are both roots of multiplicity $n$ of $L_n$. Since the degree of the polynomial $\left(X^2-1\right)^m$ is $2m$ and its leading coefficient is $1$, we have
\begin{flalign*}
I_{2m,n-m}=(2m)!\int_{-1}^{1}\dfrac{d^{n-m}}{dx^{n-m}}\left[\left(x^2-1\right)^n\right]dx=(2m)!\left[\dfrac{d^{n-m-1}}{dx^{n-m-1}}\left[\left(x^2-1\right)^n\right]\right]_{-1}^{1}=0,&&
\end{flalign*}
since $n-m-1\geq0$. Consequently, $\forall m,n\in\mathbb{N}$ such that $m\neq n,\langle L_m, L_n\rangle=0$. We finally conclude that Legendre polynomials are orthogonal for $\langle\cdot,\cdot\rangle$.\indent$\blacksquare$
\end{proof}
\begin{remark}
Legendre polynomials are, by construction, orthogonal for $\langle\cdot,\cdot\rangle$ but we are formally proving it for the sake of clarity.
\end{remark}
Using Legendre polynomials, we want to define an orthonormal basis of $\mathbb{R}[X]$ for $\langle\cdot,\cdot\rangle_l$, where $l\in\mathbb{R}_+^*$.
\begin{theorem}
Let $m\in\mathbb{N},l\in\mathbb{R}_+^*$ and $e_{m,l}=\sqrt{2m+1}L_m\left(\dfrac{X}{l}\right)$. The set of polynomials $(e_{m,l})_{m\in\mathbb{N}}$ is an orthonormal basis of $\mathbb{R}[X]$ for $\langle\cdot,\cdot\rangle_{l}$.
\end{theorem}
\begin{proof}
Let $l\in\mathbb{R}_+^*$. To prove that the set $(e_{m,l})_{m\in\mathbb{N}}$ is an orthonormal basis of $\mathbb{R}[X]$ for $\langle\cdot,\cdot\rangle_{l}$, we first need to prove that it is orthonormal for $\langle\cdot,\cdot\rangle_{l}$, then we need to prove that it is a basis of $\mathbb{R}_{m}[X]$, $\forall m\in\mathbb{N}$.
\newline 
Let $m,n\in\mathbb{N}$ such that $m\neq n$. We have
\begin{flalign*}
\langle e_{m,l},e_{n,l}\rangle_l=\dfrac{1}{2l}\int_{-l}^{l}e_{m,l}(x)e_{n,l}(x)dx=\dfrac{\sqrt{2m+1}\sqrt{2n+1}}{2l}\int_{-l}^{l}L_m(\dfrac{x}{l})L_n(\dfrac{x}{l})dx.&&
\end{flalign*}
We can use the substitution $t=\dfrac{x}{l}$ since $x\colon\mapsto\dfrac{x}{l}$ is a diffeomorphism from $[-l,l]$ to $[-1,1]$. Therefore, we obtain
\begin{flalign*}
\begin{split}
\langle e_{m,l},e_{n,l}\rangle_l&=\dfrac{\sqrt{2m+1}\sqrt{2n+1}}{2l}\int_{-1}^{1}L_m(t)L_n(t)ldt=\dfrac{\sqrt{2m+1}\sqrt{2n+1}}{2}\int_{-1}^{1}L_m(t)L_n(t)dt\\
&=\dfrac{\sqrt{2m+1}\sqrt{2n+1}}{2}\langle L_m,L_n\rangle=0.
\end{split}&&
\end{flalign*}
This proves that $(e_{m,l})_{m\in\mathbb{N}}$ is an orthogonal set. Let us show that it is orthonormal. To do so, we need to prove that $\langle e_{m,l},e_{m,l}\rangle_l=1$. We have
\begin{flalign*}
\langle e_{m,l},e_{m,l}\rangle_l=\dfrac{2m+1}{2}\langle L_m,L_m\rangle=\dfrac{2m+1}{2^{2m+1}(m!)^2}\int_{-1}^{1}\dfrac{d^m}{dx^m}\left[(x^2-1)^2\right]\dfrac{d^m}{dx^m}\left[(x^2-1)^2\right]dx=\dfrac{2m+1}{2^{2m+1}(m!)^2}I_{m,m},&&
\end{flalign*}
where
\begin{flalign*}
I_{m,m}=\int_{-1}^{1}\dfrac{d^m}{dx^m}\left[(x^2-1)^2\right]\dfrac{d^m}{dx^m}\left[(x^2-1)^2\right]dx.&&
\end{flalign*}
By iteratively using integration by parts, we obtain
\begin{flalign*}
\begin{split}
I_{m,m} &= (-1)^mI_{2m,0}=(-1)^m(2m)!\int_{-1}^1(x^2-1)^mdx=(-1)^m(2m)!\int_{-1}^1(x-1)^m(x+1)^mdx\\
&=(-1)^m(2m)!J_{m,0},
\end{split}&&
\end{flalign*}
where
\begin{flalign*}
J_{m,0}=\int_{-1}^1(x-1)^m(x+1)^mdx.&&
\end{flalign*}
By using integration by parts, we have
\begin{flalign*}
J_{m,0}=\left[(x-1)^m\dfrac{(x+1)^{m+1}}{m+1}\right]_{-1}^{1}-\dfrac{m}{m+1}\int_{-1}^{1}(x-1)^{m-1}(x+1)^{m+1}dx=-\dfrac{m}{m+1}J_{m,1}.&&
\end{flalign*}
By iterating, we find that
\begin{flalign}\label{eq:J}
\begin{split}
J_{m,0}&=(-1)^m\dfrac{(m!)^2}{(2m)!}J_{m,m}=(-1)^m\dfrac{(m!)^2}{(2m)!}\int_{-1}^{1}(x+1)^{2m}dx=(-1)^m\dfrac{(m!)^2}{(2m)!}\left[\dfrac{(x+1)^{2m+1}}{2m+1}\right]_{-1}^1\\
&=(-1)^m2^{2m+1}\dfrac{(m!)^2}{(2m+1)!}.
\end{split}&&
\end{flalign}
Finally, we obtain
\begin{flalign*}
\langle e_{m,l}, e_{m,l}\rangle_l=\dfrac{2m+1}{2^{2m+1}(m!)^2}I_{m,m}=(-1)^m\dfrac{(2m+1)!}{2^{2m+1}(m!)^2}J_{m,0}=(-1)^m\dfrac{(2m+1)!}{2^{2m+1}(m!)^2}(-1)^m2^{2m+1}\dfrac{(m!)^2}{(2m+1)!}=1.&&
\end{flalign*}
As a result, $(e_{m,l})_{m\in\mathbb{N}}$ is an orthonormal set for $\langle\cdot,\cdot\rangle_l$. In particular, $\forall m\in\mathbb{N},(e_{0,l},...,e_{m,l})$ is an orthonormal set of $\mathbb{R}_m[X]$ for $\langle\cdot,\cdot\rangle_l$ and its size is $m+1=dim(\mathbb{R}_m[X])$. Since $(e_{0,l},...,e_{m,l})$ is an orthonormal set, it is linearly independent. Hence, it is a basis of $\mathbb{R}_m[X]$. Finally, we conclude that the set of polynomials $(e_{m,l})_{m\in\mathbb{N}}$ is an orthonormal basis of $\mathbb{R}[X]$ for $\langle\cdot,\cdot\rangle_{l}$.\indent$\blacksquare$
\end{proof}
After determining the orthonormal basis, we want to determine how any polynomial expressed in the canonical basis can be expressed in the orthonormal basis.
\begin{lemma}\label{lemma:same_parity}
Let $l\in\mathbb{R}_+^*$ and $m\in\mathbb{N}$. The degrees of all the monomials of $e_{m,l}$ have the same parity as $m$.
\end{lemma}
\begin{proof}
Let $l\in\mathbb{R}_+^*,m\in\mathbb{N}$. According to the binomial theorem, $(X^2-1)^m=\sum\limits_{n=0}^m\binom{m}{n}X^{2n}(-1)^{m-n}.$ Thus, every monomial of $(X^2-1)^m$ has an even degree. Therefore, given that $L_m\propto\dfrac{d^m}{dX^m}\left[(X^2-1)^m\right]$, the degrees of every monomial of $L_m$ will either be even or odd, depending on the parity of $m$, since $deg(L_m)=m$. Given that $e_{m,l}=\sqrt{2m+1}L_m\left(\dfrac{X}{l}\right)$, the degrees of all the monomials of $e_{m,l}$ also have the same parity as $m$.\indent$\blacksquare$
\end{proof}
\begin{theorem}\label{theorem:canonical to orthonormal}
Let $l\in\mathbb{R}_+^*$.
\begin{flalign}\label{eq:canonical to orthonormal}
\forall m,n\in\mathbb{N},
\begin{cases}
\langle X^{2n}, e_{2m+1,l}\rangle_l &= \langle X^{2n+1}, e_{2m,l}\rangle_l = 0\\
\langle X^{2n}, e_{2m,l}\rangle_l &= \sqrt{4m+1}.2^{2m+1}l^{2n}\dfrac{(2n)!(m+n+1)!}{(n-m)!(2(m+n+1))!}\boldsymbol{1}_{\mathbb{N}\setminus[\![0,m[\![}(n)\\
\langle X^{2n+1}, e_{2m+1,l}\rangle_l &= \sqrt{4m+3}.2^{2m+1}l^{2n+1}\dfrac{(2n+1)!(m+n+1)!}{(n-m)!(2(m+n+1)+1)!}\boldsymbol{1}_{\mathbb{N}\setminus[\![0,m[\![}(n)
\end{cases},&&
\end{flalign}
where $\boldsymbol{1}_{\mathbb{N}\setminus[\![0,m[\![}$ is the indicator function on $\mathbb{N}\setminus[\![0,m[\![$. 
\end{theorem}
\begin{proof}
Let $l\in\mathbb{R}_+^*$ and $m,n\in\mathbb{N}$. According to Lemma \ref{lemma:same_parity}, we know that the degrees of every monomial of $e_{2m+1,l}$ is odd, since $2m+1$ is odd. We also know that
\begin{flalign*}
\forall i\in[\![0,m]\!], \langle X^{2n},X^{2i+1}\rangle_l=\dfrac{1}{2l}\int_{-l}^{l}x^{2n}x^{2i+1}dx=\dfrac{1}{2l}\int_{-l}^{l}x^{2(n+i)+1}dx=\left[\dfrac{x^{2(n+i+1)}}{2(n+i+1)}\right]_{-l}^{l}=0.&&
\end{flalign*}
Accordingly, since the degrees of every monomial of $e_{2m+1,l}$ are odd, we obtain $\langle X^{2n},e_{2m+1,l}\rangle_l=0$. The same can be deduced for $\langle X^{2n+1},e_{2m,l}\rangle_l$.
\newline
Let us determine $\langle X^{2n}, e_{2m,l}\rangle_l$. We have
\begin{flalign*}
\langle X^{2n}, e_{2m,l}\rangle_l=\dfrac{1}{2l}\int_{-l}^{l}x^{2n}e_{2m,l}(x)dx=\dfrac{\sqrt{4m+1}}{2l}\int_{-l}^{l}x^{2n}L_{2m}\left(\dfrac{x}{l}\right)dx.&&
\end{flalign*}
We can use the substitution $t=\dfrac{x}{l}$ since $x\colon\mapsto\dfrac{x}{l}$ is a diffeomorphism from $[-l,l]$ to $[-1,1]$. As a result, we obtain
\begin{flalign*}
\langle X^{2n}, e_{2m,l}\rangle_l=\dfrac{\sqrt{4m+1}}{2l}\int_{-1}^{1}\left(lt\right)^{2n}L_{2m}\left(t\right)ldt=\dfrac{\sqrt{4m+1}}{2}l^{2n}\int_{-1}^{1}t^{2n}L_{2m}\left(t\right)dt=\dfrac{\sqrt{4m+1}}{2}l^{2n}\langle X^{2n},L_{2m}\rangle.&&
\end{flalign*}
We have
\begin{flalign}\label{eq:X2nL2m}
\langle X^{2n},L_{2m}\rangle=\dfrac{1}{2^{2m}(2m)!}\int_{-1}^{1}x^{2n}\dfrac{d^{2m}}{dx^{2m}}\left[(x^2-1)^{2m}\right]dx=\dfrac{1}{2^{2m}(2m)!}K_{2n,2m},&&
\end{flalign}
where
\begin{flalign*}
K_{2n,2m}=\int_{-1}^{1}x^{2n}\dfrac{d^{2m}}{dx^{2m}}\left[(x^2-1)^{2m}\right]dx.&&
\end{flalign*}
By using integration by parts, we obtain
\begin{flalign*}
\begin{split}
K_{2n,2m}&=\left[x^{2n}\dfrac{d^{2m-1}}{dx^{2m-1}}\left[(x^2-1)^{2m}\right]\right]_{-1}^{1}-2n\int_{-1}^{1}x^{2n-1}\dfrac{d^{2m-1}}{dx^{2m-1}}\left[(x^2-1)^{2m}\right]dx\\
&=-2nK_{2n-1,2m-1}.
\end{split}&&
\end{flalign*}
If $n<m$, we can iterate the previous result $2n$ times to find that
\begin{flalign*}
K_{2n,2m}=(-1)^{2n}(2n)!K_{0,2(m-n)}=(2n)!\int_{-1}^{1}\dfrac{d^{2(m-n)}}{dx^{2(m-n}}\left[(x^2-1)^{2m}\right]dx=0.&&
\end{flalign*}
Hence, $\langle X^{2n},e_{2m,l}\rangle_l=0$ when $n<m$. If $n\geq m$, we can iterate the integration by part of $K_{2n,2m}$ $2m$ times to obtain
\begin{flalign}\label{eq:K}
\begin{split}
K_{2n,2m}&=(-1)^{2m}\dfrac{(2n)!}{(2(n-m))!}K_{2(n-m),0}=\dfrac{(2n)!}{(2(n-m))!}\int_{-1}^{1}x^{2(n-m)}\left(x^2-1\right)^{2m}dx\\
&=\dfrac{(2n)!}{(2(n-m))!}H_{n-m,2m},
\end{split}&&
\end{flalign}
where
\begin{flalign*}
H_{n-m,2m}=\int_{-1}^{1}x^{2(n-m)}\left(x^2-1\right)^{2m}dx=\int_{-1}^{1}x^{2(n-m)-1}x\left(x^2-1\right)^{2m}dx.&&
\end{flalign*}
By using integration by parts, we obtain
\begin{flalign}\label{eq:H}
\begin{split}
H_{n-m,2m}&=\left[\dfrac{x^{2(n-m)-1}\left(x^2-1\right)^{2m+1}}{2(2m+1)}\right]_{-1}^{1}-\dfrac{2(n-m)-1}{2(2m+1)}\int_{-1}^{1}x^{2(n-m-1)}\left(x^2-1\right)^{2m+1}dx\\
&=-\dfrac{2(n-m)-1}{2(2m+1)}H_{n-m-1,2m+1}=(-1)^{n-m}\dfrac{(2(n-m))!}{2^{n-m}(n-m)!}\dfrac{(2m)!}{2^{n-m}(m+n)!}H_{0,m+n}\\
&=(-1)^{n-m}\dfrac{(2(n-m))!}{2^{2(n-m)}(n-m)!}\dfrac{(2m)!}{(m+n)!}\int_{-1}^{1}(x^2-1)^{m+n}dx\\
&=(-1)^{n-m}\dfrac{(2(n-m))!}{2^{2(n-m)}(n-m)!}\dfrac{(2m)!}{(m+n)!}J_{m+n,0}.
\end{split}&&
\end{flalign}
Using Eqs.(\ref{eq:J}), (\ref{eq:H}), (\ref{eq:K}), and (\ref{eq:X2nL2m}), we find that
\begin{flalign*}
\begin{split}
\langle X^{2n},e_{2m,l}\rangle_l &=\dfrac{\sqrt{4m+1}}{2}l^{2n}\dfrac{(2n)!(-1)^{n-m}(2(n-m))!(2m)!(-1)^{m+n}2^{2(m+n)+1}\left((m+n)!\right)^2}{2^{2m}(2m)!(2(n-m))!2^{2(n-m)}(n-m)!(m+n)!((2(m+n)+1)!}\\
&= \sqrt{4m+1}.2^{2m+1}l^{2n}\dfrac{(2n)!(m+n+1)!}{(n-m)!(2(m+n+1))!}.
\end{split}&&
\end{flalign*}
Since $\langle X^{2n},e_{2m,l}\rangle_l=0$ when $n<m$, we finally conclude that
\begin{flalign*}
\langle X^{2n}, e_{2m,l}\rangle_l &= \sqrt{4m+1}.2^{2m+1}l^{2n}\dfrac{(2n)!(m+n+1)!}{(n-m)!(2(m+n+1))!}\boldsymbol{1}_{\mathbb{N}\setminus[\![0,m[\![}(n).&&
\end{flalign*}
The same steps can be followed to determine $\langle X^{2n+1}, e_{2m+1,l}\rangle_l$.\indent$\blacksquare$
\end{proof}
Using Theorem \ref{theorem:canonical to orthonormal}, we can fully determine $T^{[M,N,l]}$ defined in Eq. (\ref{eq:TMN}) such that $\forall (m,n)\in[\![0,M]\!]\times[\![0,N]\!],T^{[M,N,l]}_{m,n}=\langle X^n,e_{m,l}\rangle_l.$
\section{Determining the orthonormal projection of $\mathbb{R}_N[X]$ on $\mathbb{R}_M[X]$ in the canonical basis}\label{section:canonical basis}
In this section, we use the theorems proposed in section \ref{section:orthonormal basis} to solve the problem defined in section \ref{section:problem statement}.
In order to determine the canonical basis coefficients, we need to determine the transition matrix $\left(T^{[M,M,l]}\right)^{-1}$ from the orthonormal basis to the canonical basis as defined in Eq. (\ref{eq:TMN}). To do so, we need to express the set of orthonormal polynomials $\left(e_{m,l}\right)_{m\in\mathbb{N}}$ in the canonical basis.
\begin{theorem}\label{theorem:orthonormal to canonical}
Let $l\in\mathbb{R}_+^*$. $\forall m\in\mathbb{N},e_{m,l}=\sum\limits_{n=0}^{m}\epsilon_{m,n,l}X^n$ where
\begin{flalign}\label{eq:orthonormal to canonical}
\forall m,n\in\mathbb{N},
\begin{cases}
\epsilon_{2m,2n+1,l}&=\epsilon_{2m+1,2n,l}=0\\
\epsilon_{2m,2n,l}&=(-1)^{m-n}\dfrac{\sqrt{4m+1}}{2^{2m}l^{2n}}\dfrac{(2(m+n))!}{(m+n)!(m-n)!(2n)!}\boldsymbol{1}_{\mathbb{N}\setminus[\![0,n[\![}(m)\\
\epsilon_{2m+1,2n+1,l}&=(-1)^{m-n}\dfrac{\sqrt{4m+3}}{2^{2m}l^{2n}}\dfrac{(2(m+n)+1)!}{(m+n)!(m-n)!(2n+1)!}\boldsymbol{1}_{\mathbb{N}\setminus[\![0,n[\![}(m)
\end{cases}
.&&
\end{flalign}
\end{theorem}
\begin{proof}
Let $l\in\mathbb{R}_+^*,m\in\mathbb{N}$. We have 
\begin{flalign*}
\begin{split}
e_{2m,l}&=\sqrt{4m+1}L_{2m}\left(\dfrac{X}{l}\right)=\dfrac{\sqrt{4m+1}}{2^{2m}(2m)!}\dfrac{d^{2m}}{d\left(\dfrac{X}{l}\right)^{2m}}\left[\left(\left(\dfrac{X}{l}\right)^2-1\right)^{2m}\right]\\
&=\dfrac{\sqrt{4m+1}l^{2m}}{2^{2m}(2m)!}\dfrac{d^{2m}}{dX^{2m}}\left[\left(\left(\dfrac{X}{l}\right)^2-1\right)^{2m}\right].
\end{split}&&
\end{flalign*}
According to the binomial theorem, we have
\begin{flalign*}
\left(\left(\dfrac{X}{l}\right)^2-1\right)^{2m}=\sum_{n=0}^{2m}\binom{2m}{n}\left(\dfrac{X}{l}\right)^{2n}(-1)^{2m-n}.&&
\end{flalign*}
Consequently, we obtain
\begin{flalign*}
\begin{split}
e_{2m,l}&=\dfrac{\sqrt{4m+1}l^{2m}}{2^{2m}(2m)!}\sum_{n=m}^{2m}\binom{2m}{n}\dfrac{(-1)^{2m-n}(2n)!}{(2(n-m))!l^{2n}}X^{2(n-m)}\\
&=\dfrac{\sqrt{4m+1}}{2^{2m}}\sum_{n=m}^{2m}(-1)^{2m-n}\dfrac{1}{l^{2(n-m)}}\dfrac{(2n)!}{n!(2m-n)!((2(n-m))!}X^{2(n-m)}.
\end{split}&&
\end{flalign*}
By using the index substitution $n\leftarrow n-m$, we have
\begin{flalign*}
e_{2m,l}=\sum_{n=0}^m(-1)^{m-n}\dfrac{\sqrt{4m+1}}{2^{2m}l^{2m}}\dfrac{(2(m+n))!}{(m+n)!(m-n)!(2n)!}X^{2n}=\sum_{n=0}^{m}\epsilon_{2m,2n,l}X^{2n}.&&
\end{flalign*}
Therefore, by identification, we have $\epsilon_{2m, 2n+1,l}=0$ and 
\begin{flalign*}
\epsilon_{2m,2n,l}=
\begin{cases}
0&, \text{if $n>m$}\\
(-1)^{m-n}\dfrac{\sqrt{4m+1}}{2^{2m}l^{2m}}\dfrac{(2(m+n))!}{(m+n)!(m-n)!(2n)!}&,\text{otherwise}
\end{cases}.&&
\end{flalign*}
Finally, we obtain
\begin{flalign*}
\epsilon_{2m,2n,l}&=(-1)^{m-n}\dfrac{\sqrt{4m+1}}{2^{2m}l^{2m}}\dfrac{(2(m+n))!}{(m+n)!(m-n)!(2n)!}\boldsymbol{1}_{\mathbb{N}\setminus[\![0,n[\![}(m).&&
\end{flalign*}
The same steps can be followed for $\epsilon_{2m+1,2n,l}$ and $\epsilon_{2m+1,2n+1,1}$.\indent$\blacksquare$
\end{proof}
We can now use Theorem \ref{theorem:orthonormal to canonical} to determine $\left(T^{[M,M,l]}\right)^{-1}$ such that $\forall (m,n)\in[\![0,M]\!]^2,\left(T^{[M,M,l]}\right)^{-1}_{m,n}=\epsilon_{n,m}$. Having determined $T^{[M,N,l]}$ and $\left(T^{[M,M,l]}\right)^{-1}$, the canonical coefficients can be determined by calculating $V^{[M,N,l]}=\left(T^{[M,M,l]}\right)^{-1}T^{[M,N,l]}$. We can notice that $V^{[M,M,l]}=I_{M}$ where $I_{M}$ is the identity matrix. In that case, there is only the need to determine the elements of $V^{[M,N,l]}$ whose columns are strictly greater than $M$. We denote by $v_{m,n}^{[M,N,l]}$ an element of $V^{[M,N,l]}$ at line $m$ and column $n$, $\forall (m,n)\in[\![0,M]\!]\times[\![0,N]\!]$.
\begin{theorem}\label{theorem:V}
Let $l\in\mathbb{R}_+^*$ and $M,N\in\mathbb{N}$ such that $M<N$. Let $p=\left\lfloor\dfrac{M}{2}\right\rfloor, p_1=\left\lfloor\dfrac{M-1}{2}\right\rfloor,q=\left\lfloor\dfrac{N}{2}\right\rfloor$ and $q_1=\left\lfloor\dfrac{N-1}{2}\right\rfloor$. We have
\begin{flalign*}
\begin{split}
&\forall (m,n)\in[\![0,p]\!]\times[\![p_1+1,q_1]\!], v_{2m,2n+1}^{[M,N,l]}=0,\\
&\forall (m,n)\in[\![0,p_1]\!]\times[\![p+1,q]\!],v_{2m+1,2n}^{[M,N,l]}=0,\\
&\forall (m,n)\in[\![0,p]\!]\times[\![p+1,q]\!],v_{2m,2n}^{[M,N,l]}=\dfrac{(-1)^{p-m}l^{2(n-m)}(2n)!}{2^{n-m}(n-m)!(p-m)!(2m)!}\prod\limits_{r=n-p}^{n-m-1}r\prod\limits_{r=p+m+1}^{p+n}\dfrac{1}{2r+1},\\
&\text{and}\\
&\forall (m,n)\in[\![0,p_1]\!]\times[\![p_1+1,q_1]\!],v_{2m+1,2n+1}^{[M,N,l]}=\dfrac{(-1)^{p-m}l^{2(n-m)}(2n+1)!}{2^{n-m}(n-m)!(p-m)!(2m+1)!}\prod\limits_{r=n-p}^{n-m-1}r\prod\limits_{r=p+m+2}^{p+n+1}\dfrac{1}{2r+1}.
\end{split}&&
\end{flalign*}
\end{theorem}
\begin{proof}
Let $l\in\mathbb{R}_+^*$ and $M,N\in\mathbb{N}$ such that $M<N$. Let $p=\left\lfloor\dfrac{M}{2}\right\rfloor, p_1=\left\lfloor\dfrac{M-1}{2}\right\rfloor,q=\left\lfloor\dfrac{N}{2}\right\rfloor$ and $q_1=\left\lfloor\dfrac{N-1}{2}\right\rfloor$. We will first prove that $\forall (m,n)\in[\![0,p]\!]\times[\![p_1+1,q_1]\!], v_{2m,2n+1}^{[M,N,l]}=0$. 
\newline
Let $(m,n)\in[\![0,p]\!]\times[\![p_1+1,q_1]\!]$. We have
\begin{flalign*}
v_{2m,2n+1}^{[M,N,l]}=\sum_{k=0}^{M}\epsilon_{k,2m,l}\langle X^{2n+1},e_{k,l}\rangle_l.&&
\end{flalign*}
According to Theorem \ref{theorem:canonical to orthonormal} and Theorem \ref{theorem:orthonormal to canonical}, $\langle X^{2n+1},e_{k,l}\rangle_l=0$ when $k$ is even and $\epsilon_{k,2m,l}=0$ when $k$ is odd. However, $k$ can not be even and odd at the same time, thus, one of both terms in the multiplication will necessarily be nil. Thus, $v_{2m,2n+1}^{[M,N,l]}=0$. The same reasoning can be used for $v_{2m+1,2n}^{[M,N,l]}$.
\newline 
We now want to determine the value of $v_{2m,2n}^{[M,N,l]}, \forall (m,n)\in[\![0,p]\!]\times[\![p+1,q]\!]$. Let $(m,n)\in[\![0,p]\!]\times[\![p+1,q]\!]$. We have
\begin{flalign*}
v_{2m,2n}^{[M,N,l]}=\sum_{k=0}^{M}\epsilon_{k,2m,l}\langle X^{2n},e_{k,l}\rangle_l.&&
\end{flalign*}
According to Theorem \ref{theorem:canonical to orthonormal} and Theorem \ref{theorem:orthonormal to canonical}, $\langle X^{2n},e_{k,l}\rangle_l=0$ when $k$ is odd and $\epsilon_{k,2m,l}=0$ when $k$ is odd. Hence we obtain
\begin{flalign*}
v_{2m,2n}^{[M,N,l]}=\sum_{k=0}^{p}\epsilon_{2k,2m,l}\langle X^{2n},e_{2k,l}\rangle_l.&&
\end{flalign*}
According to Theorem \ref{theorem:orthonormal to canonical}, $\epsilon_{2k,2m,l}=0$ when $k<m$. As a result, we have
\begin{flalign*}
v_{2m,2n}^{[M,N,l]}=\sum_{k=m}^{p}\epsilon_{2k,2m,l}\langle X^{2n},e_{2k,l}\rangle_l.&&
\end{flalign*}
In order to determine $v_{2m,2n}^{[M,N,l]}$, we can first determine $v_{2(p-m),2n}^{[M,N,l]}$ as such:
\begin{flalign}\label{eq:v_2(p-m) before simplification}
\begin{split}
v_{2(p-m),2n}^{[M,N,l]}&=\sum_{k=p-m}^{p}\epsilon_{2k,2(p-m),l}\langle X^{2n},e_{2k,l}\rangle_l\\
&=\sum_{k=p-m}^{p}\dfrac{(-1)^{k+m-p}2(4k+1)l^{2(n+m-p)}}{(k+m-p)!(n-k)!}\dfrac{(2n)!}{(2(p-m))!}\dfrac{(k+n+1)!}{(2(k+n+1))!}\dfrac{(2(k+p-m))!}{(k+p-m)!}.
\end{split}&&
\end{flalign}
We have
\begin{flalign*}
(2(k+p-m))! = \prod\limits_{r=1}^{2(k+p-m)}r=\prod\limits_{r=1}^{k+p-m}2r\prod\limits_{r=0}^{k+p-m-1}(2r+1)=2^{k+p-m}(k+p-m)!\prod\limits_{r=0}^{k+p-m-1}(2r+1).&&
\end{flalign*}
Thus,
\begin{flalign}\label{eq:k+p-m}
\dfrac{(2(k+p-m))!}{(k+p-m)!}=2^{k+p-m}\prod\limits_{r=0}^{k+p-m-1}(2r+1).&&
\end{flalign}
The same steps can be followed to show that
\begin{flalign}\label{eq:k+n+1}
\dfrac{(k+n+1)!}{(2(k+n+1))!}=\dfrac{1}{2^{k+n+1}}\prod\limits_{r=0}^{k+n}\dfrac{1}{2r+1}.&&
\end{flalign}
From Eqs. (\ref{eq:k+p-m}) and (\ref{eq:k+n+1}), we can derive the following expression: 
\begin{flalign}\label{eq:(k+p-m)/(k+n+1)}
\dfrac{(k+n+1)!}{(2(k+n+1))!}\dfrac{(2(k+p-m))!}{(k+p-m)!}=\dfrac{1}{2^{n+m-p+1}}\prod\limits_{r=k+p-m}^{k+n}\dfrac{1}{2r+1}.&&
\end{flalign}
By injecting Eq. (\ref{eq:(k+p-m)/(k+n+1)}) in Eq. (\ref{eq:v_2(p-m) before simplification}), and by factorizing the terms of the summation that do not depend on $k$, we obtain
\begin{flalign}\label{eq:v_2(p-m) after simplification}
v_{2(p-m),2n}^{[M,N,l]}=\dfrac{l^{(2(n+m-p)}(2n)!}{2^{n+m-p}(2(p-m))!}\sum_{k=p-m}^{p}(-1)^{k-p+m}\dfrac{(4k+1)}{(k+m-p)!(n-k)!}\prod\limits_{r=k+p-m}^{k+n}\dfrac{1}{2r+1}.&&
\end{flalign}
Since Eq. (\ref{eq:v_2(p-m) after simplification}) consists in the summation of different fractions, we express the fractions with respect to the same common denominator as such:
\begin{flalign}\label{eq:v_2(p-m) with respect to Smnp}
v_{2(p-m),2n}^{[M,N,l]}=\dfrac{l^{2(n+m-p)}(2n)!}{2^{n+m-p}(2(p-m))!(n-p+m)!m!}\prod\limits_{r=2(p-m)}^{p+n}\dfrac{1}{2r+1}S_{m,n,p},&&
\end{flalign}
where
\begin{flalign*}
S_{m,n,p}=\sum_{k=p-m}^{p}(-1)^{k+m-p}(4k+1)\prod\limits_{n-k+1}^{n-p+m}r\prod\limits_{k+m-p+1}^{m}r\prod\limits_{2(p-m)}^{k+p-m-1}(2r+1)\prod\limits_{k+n+1}^{p+n}(2r+1).&&
\end{flalign*}
The objective is now to prove that 
\begin{flalign*}
S_{m,n,p}=(-1)^{m}\prod\limits_{2(p-m)}^{2p-m}(2r+1)\prod\limits_{n-p}^{n-p+m-1}r.&&
\end{flalign*}
To do so, we first need to rearrange the terms of $S_{m,n,p}$. We use the index substitution $k\leftarrow p-k$ to obtain
\begin{flalign*}
S_{m,n,p}=\sum_{k=0}^{m}(-1)^{m-k}(4(p-k)+1)\prod\limits_{r=n-p+k+1}^{n-p+m}r\prod\limits_{r=m+1-k}^{m}r\prod\limits_{r=2(p-m)}^{2p-m-k-1}(2r+1)\prod\limits_{r=p+n-k+1}^{p+n}(2r+1).&&
\end{flalign*}
We then incorporate the common terms of each boundary of each product into each product index to obtain
\begin{flalign*}
S_{m,n,p}=\sum_{k=0}^{m}(-1)^{k+m}(4(p-k)+1)\prod\limits_{r=k+1}^{m}(r+n-p)\hspace{-.5em}\hspace{-.5em}\prod\limits_{r=-k+m+1}^{m}\hspace{-.5em}\hspace{-.5em}\hspace{-.5em}r\hspace{.5em}\prod\limits_{r=-2m}^{-k-m-1}(4p+2r+1)\hspace{-.5em}\prod\limits_{r=-k+1}^{0}(2(r+p+n)+1).&&
\end{flalign*}
Let $S_m(X,Y)\in\mathbb{R}[X,Y]$ be a polynomial in 2 indeterminates $X$ and $Y$ over $\mathbb{R}$ defined as such:
\begin{flalign*}
S_m=\sum_{k=0}^{m}(-1)^{k+m}(4(Y-k)+1)\prod\limits_{r=k+1}^{m}(r+X-Y)\hspace{-.5em}\hspace{-.5em}\prod\limits_{r=-k+m+1}^{m}\hspace{-.5em}\hspace{-.5em}\hspace{-.5em}r\hspace{.5em}\prod\limits_{r=-2m}^{-k-m-1}(4Y+2r+1)\hspace{-.5em}\prod\limits_{r=-k+1}^{0}(2(r+Y+X)+1).&&
\end{flalign*}
This makes $S_{m,n,p}=S_m(n,p)$. Since our objective is to show that 
\begin{flalign*}
S_{m,n,p}=(-1)^{m}\prod\limits_{2(p-m)}^{2p-m}(2r+1)\prod\limits_{n-p}^{n-p+m-1}r=(-1)^{m}\prod\limits_{r=0}^{m}(4(p-m)+2r+1)\prod\limits_{r=0}^{m-1}(r+n-p),&&
\end{flalign*}
we can show that the polynomial $S_m(n,X)$ in one indeterminate $X$ can be factorized as such:
\begin{flalign*}
S_m(n,X)=(-1)^{m}\prod\limits_{r=0}^{m}(4(X-m)+2r+1)\prod\limits_{r=0}^{m-1}(r+n-X).&&
\end{flalign*}
Therefore, the objective is to show that $\forall t\in[\![0,m]\!],m-\dfrac{1}{4}(2t+1)$ is a root of $S_m(n,X)$ and that $\forall t\in[\![0,m-1]\!],t+n$ is also a root of $S_m(n,X)$. We begin by showing that $m-\dfrac{1}{4}(2t+1)$ is a root of $S_m(n,X),\forall t\in[\![0,m]\!]$. Let $t\in[\![0,m]\!]$. We want to show that $S_m\left(n,m-\dfrac{1}{4}(2t+1)\right)=0$. We have
\begin{flalign*}
\begin{split}
S_m\left(n,m-\dfrac{1}{4}(2t+1)\right)=\sum_{k=0}^{m}&(-1)^{k+m}(4(m-k)-2t)\prod\limits_{r=k+1}^{m}\left(r+n-m+\dfrac{t}{2}+\dfrac{1}{4}\right)\prod\limits_{r=-k+m+1}^{m}r\\
&\prod\limits_{r=-2m}^{-k-m-1}(4m+2(r-t))\prod\limits_{r=-k+1}^{0}\left(2(r+n+m)-t+\dfrac{1}{2}\right)\\
=\sum_{k=0}^{m}&U_{m,n,t,k},
\end{split}&&
\end{flalign*}
where $U_{m,n,t,k}$ represents the k-th term of the summation. The term $\prod\limits_{r=-2m}^{-k-m-1}(4m+2(r-t))$ is nil when $\exists r'\in\mathbb{N}$ such that $-2m\leq r'\leq -k-m-1$ and $4m+2(r'-t)=0$. This means that $r'=-2m+t$. However, since $r'\leq -k-m-1$, we find that, when $k\leq m-t-1$, $\prod\limits_{r=-2m}^{-k-m-1}(4m+2(r-t))=0$. Moreover, the term $4(m-k)-2t$ is nil when $k=m-\dfrac{t}{2}$. This means that $4(m-k)-2t=0$ when $t$ is even. We suppose that $t$ is even. Then, we can rewrite $S_m\left(n,m-\dfrac{1}{4}(2t+1)\right)$ by splitting the sum in two parts as such:
\begin{flalign}\label{eq:S split in two}
S_m\left(n,m-\dfrac{1}{4}(2t+1)\right)=\sum_{k=m-t}^{m-\left\lfloor \tfrac{t}{2}\right\rfloor-1}U_{m,n,t,k}+\sum_{k=m-\left\lfloor \tfrac{t}{2}\right\rfloor+1}^{m}U_{m,n,t,k}.&&
\end{flalign}
By using the substitution $k\leftarrow m-\left\lfloor \dfrac{t}{2}\right\rfloor-k$ on the first sum of Eq. (\ref{eq:S split in two}), and the substitution $k\leftarrow -m+\left\lfloor \dfrac{t}{2}\right\rfloor+k$ on the second sum of Eq. (\ref{eq:S split in two}), we can regroup both sums to obtain
\begin{flalign}\label{eq:S same sum index}
S_m\left(n,m-\dfrac{1}{4}(2t+1)\right)=\sum_{k=1}^{\left\lfloor \tfrac{t}{2}\right\rfloor}\left(U_{m,n,t,m-\left\lfloor \tfrac{t}{2}\right\rfloor-k}+U_{m,n,t,m-\left\lfloor \tfrac{t}{2}\right\rfloor+k}\right).&&
\end{flalign}
We have 
\begin{flalign}\label{eq:U first half}
\begin{split}
U_{m,n,t,m-\left\lfloor \tfrac{t}{2}\right\rfloor-k}=&(-1)^{2m-\left\lfloor \tfrac{t}{2}\right\rfloor-k}4k\prod\limits_{r=m-\left\lfloor \tfrac{t}{2}\right\rfloor-k+1}^{m}\left(r+n-m+\dfrac{t}{2}+\dfrac{1}{4}\right)\prod\limits_{r=k+\left\lfloor \tfrac{t}{2}\right\rfloor+1}^{m}r\\
&\prod\limits_{r=-2m}^{k+\left\lfloor \tfrac{t}{2}\right\rfloor-2m-1}(4m+2(r-t))\prod\limits_{r=k+\left\lfloor \tfrac{t}{2}\right\rfloor-m+1}^{0}\left(2(r+n+m)-t+\dfrac{1}{2}\right),
\end{split}&&
\end{flalign}
and
\begin{flalign}\label{eq:U second half}
\begin{split}
U_{m,n,t,m-\left\lfloor \tfrac{t}{2}\right\rfloor+k}=&(-1)^{2m-\left\lfloor \tfrac{t}{2}\right\rfloor+k}(-4k)\prod\limits_{r=m-\left\lfloor \tfrac{t}{2}\right\rfloor+k+1}^{m}\left(r+n-m+\dfrac{t}{2}+\dfrac{1}{4}\right)\prod\limits_{r=-k+\left\lfloor \tfrac{t}{2}\right\rfloor+1}^{m}r\\
&\prod\limits_{r=-2m}^{-k+\left\lfloor \tfrac{t}{2}\right\rfloor-2m-1}(4m+2(r-t))\prod\limits_{r=-k+\left\lfloor \tfrac{t}{2}\right\rfloor-m+1}^{0}\left(2(r+n+m)-t+\dfrac{1}{2}\right).
\end{split}&&
\end{flalign}
By factorizing the common terms of $U_{m,n,t,m-\left\lfloor \tfrac{t}{2}\right\rfloor-k}$ and $U_{m,n,t,m-\left\lfloor \tfrac{t}{2}\right\rfloor+k}$, we obtain
\begin{flalign}\label{eq:U factorized}
\begin{split}
&U_{m,n,t,m-\left\lfloor \tfrac{t}{2}\right\rfloor-k}+U_{m,n,t,m-\left\lfloor \tfrac{t}{2}\right\rfloor+k}\\
=&(-1)^{2m-\left\lfloor \tfrac{t}{2}\right\rfloor+k}4k\prod\limits_{r=m-\left\lfloor \tfrac{t}{2}\right\rfloor+k+1}^{m}\left(r+n-m+\dfrac{t}{2}+\dfrac{1}{4}\right)\prod\limits_{r=+k+\left\lfloor \tfrac{t}{2}\right\rfloor+1}^{m}r\prod\limits_{r=-2m}^{-k+\left\lfloor \tfrac{t}{2}\right\rfloor-2m-1}(4m+2(r-t))\\
&\prod\limits_{r=+k+\left\lfloor \tfrac{t}{2}\right\rfloor-m+1}^{0}\left(2(r+n+m)-t+\dfrac{1}{2}\right)W_{m,n,t,k},
\end{split}&&
\end{flalign}
where,
\begin{flalign}\label{eq:W}
\begin{split}
W_{m,n,t,k}=&\prod\limits_{r=m-\left\lfloor \tfrac{t}{2}\right\rfloor-k+1}^{m-\left\lfloor \tfrac{t}{2}\right\rfloor+k}\left(r+n-m+\dfrac{t}{2}+\dfrac{1}{4}\right)\prod\limits_{r=-k+\left\lfloor \tfrac{t}{2}\right\rfloor-2m}^{+k+\left\lfloor \tfrac{t}{2}\right\rfloor-2m-1}(4m+2(r-t))\\
&-\prod\limits_{r=-k+\left\lfloor \tfrac{t}{2}\right\rfloor+1}^{k+\left\lfloor \tfrac{t}{2}\right\rfloor}r\prod\limits_{r=-k+\left\lfloor \tfrac{t}{2}\right\rfloor-m+1}^{k+\left\lfloor \tfrac{t}{2}\right\rfloor-m}\left(2(r+n+m)-t+\dfrac{1}{2}\right).
\end{split}&&
\end{flalign}
By using the index substitutions $r\leftarrow -r$, $r\leftarrow r+m$, $r\leftarrow r-m-1$ and $r\leftarrow r-1$ respectively on the 4 products in Eq. (\ref{eq:W}), we obtain
\begin{flalign}\label{eq:W same product index}
\begin{split}
W_{m,n,t,k}=&\prod\limits_{r=-m+\left\lfloor \tfrac{t}{2}\right\rfloor-k}^{-m+\left\lfloor \tfrac{t}{2}\right\rfloor+k-1}\left(-r+n-m+ \dfrac{t}{2}+\dfrac{1}{4}\right)(2m+2(r-t))\\
&-\prod\limits_{r=-k+\left\lfloor \tfrac{t}{2}\right\rfloor-m}^{k+\left\lfloor \tfrac{t}{2}\right\rfloor-m-1}\left(r+m+1\right)\left(2(r+n+m+1)-t+\dfrac{1}{2}\right).
\end{split}
&&
\end{flalign}
By using the index substitution $r\leftarrow -2m+t-1-r$ on the second product in Eq. (\ref{eq:W same product index}), we obtain
\begin{flalign}\label{eq:second product in W}
\begin{split}
&\prod\limits_{r=-k+\left\lfloor \tfrac{t}{2}\right\rfloor-m}^{k+\left\lfloor \tfrac{t}{2}\right\rfloor-m-1}\left(r+m+1\right)\left(2(r+n+m+1)-t+\dfrac{1}{2}\right)\\
=&\prod\limits_{r=-k+\left\lfloor \tfrac{t}{2}\right\rfloor-m}^{k+\left\lfloor \tfrac{t}{2}\right\rfloor-m-1}(t-r-m)\left(-2r+2n-2m+t+\dfrac{1}{2}\right)\\
=&\prod\limits_{r=-k+\left\lfloor \tfrac{t}{2}\right\rfloor-m}^{k+\left\lfloor \tfrac{t}{2}\right\rfloor-m-1}(2t-2r-2m)\left(-r+n-m+\dfrac{t}{2}+\dfrac{1}{4}\right)\\
=&\prod\limits_{r=-k+\left\lfloor \tfrac{t}{2}\right\rfloor-m}^{k+\left\lfloor \tfrac{t}{2}\right\rfloor-m-1}(-(2m+2(r-t)))\left(-r+n-m+\dfrac{t}{2}+\dfrac{1}{4}\right)\\
=&(-1)^{2k}\prod\limits_{r=-k+\left\lfloor \tfrac{t}{2}\right\rfloor-m}^{k+\left\lfloor \tfrac{t}{2}\right\rfloor-m-1}(2m+2(r-t))\left(-r+n-m+\dfrac{t}{2}+\dfrac{1}{4}\right)\\
=&\prod\limits_{r=-k+\left\lfloor \tfrac{t}{2}\right\rfloor-m}^{k+\left\lfloor \tfrac{t}{2}\right\rfloor-m-1}(2m+2(r-t))\left(-r+n-m+\dfrac{t}{2}+\dfrac{1}{4}\right).
\end{split}&&
\end{flalign}
By injecting Eq. (\ref{eq:second product in W}) in Eq. (\ref{eq:W same product index}), we find that $W_{m,n,t,k}=0$. Consequently, using Eq. (\ref{eq:U factorized}), we have $U_{m,n,t,m-\left\lfloor \tfrac{t}{2}\right\rfloor-k}+U_{m,n,t,m-\left\lfloor \tfrac{t}{2}\right\rfloor+k}=0$. Using Eq. (\ref{eq:S same sum index}), we find that $S_m\left(n,m-\dfrac{1}{4}(2t+1)\right)=0$ when $t$ is even. The same can be concluded when $t$ is odd by rewriting $S_m\left(n,m-\dfrac{1}{4}(2t+1)\right)$ as such:
\begin{flalign}\label{S split in two odd}
S_m\left(n,m-\dfrac{1}{4}(2t+1)\right)=\sum_{k=m-t}^{m-\left\lfloor \tfrac{t}{2}\right\rfloor-1}U_{m,n,t,k}+\sum_{k=m-\left\lfloor \tfrac{t}{2}\right\rfloor}^{m}U_{m,n,t,k}.&&
\end{flalign}
The same steps can then be performed on Eq.(\ref{S split in two odd}) to show that it is nil. The only difference lies in applying the substitution $k\leftarrow -m+\left\lfloor \dfrac{t}{2}\right\rfloor+k+1$ on the second sum of Eq.(\ref{S split in two odd}). We can subsequently conclude that $m-\dfrac{1}{4}(2t+1)$ is a root of $S_m(n,X),\forall t\in[\![0,m]\!]$. Moreover, we notice that $n$ has no influence on whether $m-\dfrac{1}{4}(2t+1)$ is a root of $S_m(n,X),\forall t\in[\![0,m]\!]$, thus we can induce that $S_m\left(X,m-\dfrac{1}{4}(2t+1)\right)=0,\forall t\in[\![0,m]\!]$. Furthermore, since $p+1\leq n\leq q,\exists n'\in[\![1,q-p]\!]$ such that $n=n'+p$. We then have $S_{m,n,p}=S_{m,n'+p,p}=S_m(n'+p,p)$. Hence, we can study the polynomial $S_m(n'+X,X)$ to prove that 
\begin{flalign*}
S_m(n'+X,X)=(-1)^{m}\prod\limits_{r=-2m}^{-m}(4X+2r+1)\prod\limits_{r=0}^{m-1}(r+n').&&
\end{flalign*}
In order to do so, we need to prove that $deg\left(S_m(n'+X,X)\right)=m+1$ and that the leading coefficient of $S_m(n'+X,X)$ is $(-1)^{m}4^{m+1}\prod\limits_{r=0}^{m-1}(r+n')$. We have
\begin{flalign}\label{eq:S(n'+X,X)}
\begin{split}
S_m(n'+X,X)=\sum_{k=0}^{m}&(-1)^{k+m}(4(X-k)+1)\prod\limits_{r=k+1}^{m}(r+n')\hspace{-.5em}\hspace{-.5em}\prod\limits_{r=-k+m+1}^{m}\hspace{-.5em}\hspace{-.5em}\hspace{-.5em}r\hspace{.5em}\prod\limits_{r=-2m}^{-k-m-1}(4X+2r+1)\hspace{-.5em}\\
&\prod\limits_{r=-k+1}^{0}(2(r+n'+2X)+1).
\end{split}&&
\end{flalign}
By examining the monomials of $S_m(n'+X,X)$, we find that $deg\left(S_m(n'+X,X)\right)\leq1+(-k-m-1+2m+1)+k\leq m+1$. Since we proved that $S_m(n'+X,X)$ has at least $m+1$ distinct roots, $deg\left(S_m(n'+X,X)\right)=m+1$. To determine the leading coefficient $c_{m,n'}$ of $S_m(n'+X,X)$, we observe that every term of the summation in Eq. (\ref{eq:S(n'+X,X)}) is a polynomial of degree $m+1$. Then, we identify the leading coefficient of every term, and sum it with the other coefficients to obtain
\begin{flalign*}
\begin{split}
c_{m,n'}&=\sum_{k=0}^{m}(-1)^{k+m}4\prod\limits_{r=k+1}^{m}(r+n')\left(\prod\limits_{r=-k+m+1}^{m}r\right)4^{m-k}4^k.\\
&=(-1)^m4^{m+1}\sum_{k=0}^{m}(-1)^k\dfrac{(m+n')!}{(k+n')}\dfrac{m!}{(m-k)!}.
\end{split}&&
\end{flalign*}
By using the index substitution $k\leftarrow m-k$, we obtain
\begin{flalign}\label{eq:leading coefficient}
c_{m,n'}=4^{m+1}m!\sum_{k=0}^{m}(-1)^k\dfrac{(m+n')!}{k!(m+n'-k)!}=4^{m+1}m!\sum_{k=0}^{m}(-1)^k\binom{m+n'}{k}.&&
\end{flalign}
Henceforth, the objective is to show that 
\begin{flalign*}
\sum_{k=0}^{m}(-1)^k\binom{m+n'}{k}=(-1)^m\dfrac{1}{m!}\prod\limits_{r=0}^{m-1}(r+n')=(-1)^m\dfrac{(m+n'-1)}{m!(n'-1)!}=(-1)^m\binom{m+n'-1}{m}.&&
\end{flalign*}
To do so, we use Pascal's rule to obtain
\begin{flalign}\label{eq:sum cnk}
\begin{split}
\sum_{k=0}^{m}(-1)^k\binom{m+n'}{k}&=\sum_{k=0}^{m}(-1)^{k}\left(\binom{m+n'-1}{k}+\binom{m+n'-1}{k-1}\right)\\
&=\sum_{k=0}^{m-1}(-1)^k\binom{m+n'-1}{k}+(-1)^{m}\binom{m+n'-1}{m}+\sum_{k=0}^{m-1}(-1)^{k+1}\binom{m+n'-1}{k}\\
&=(-1)^{m}\binom{m+n'-1}{m}.
\end{split}&&
\end{flalign}
As a result, by using Eq. (\ref{eq:sum cnk}) in Eq. (\ref{eq:leading coefficient}), the leading coefficient $c_{m,n'}$ of $S_m(n'+X,X)$ is
\begin{flalign*}
c_{m,n'}=(-1)^m4^{m+1}m!\binom{m+n'-1}{m}=(-1)^m4^{m+1}\dfrac{(m+n'-1)!}{(n'-1)!}=(-1)^{m}4^{m+1}\prod\limits_{r=0}^{m-1}(r+n').&&
\end{flalign*}
Thus, since $\forall r\in[\![0,m]\!],m-\dfrac{1}{4}(2r+1)$ is a root of $S_m(n'+X,X)$, and $deg(S_m(n'+X,X))=m+1$, we have
\begin{flalign*}
\begin{split}
S_m(n'+X,X)&=(-1)^{m}4^{m+1}\prod\limits_{r=0}^{m}(X-m+\dfrac{1}{4}(2r+1))\prod\limits_{r=0}^{m-1}(r+n')\\
&=(-1)^{m}4^{m+1}\prod\limits_{r=0}^{m}\left(\dfrac{4(X-m)+2r+1}{4}\right)\prod\limits_{r=0}^{m-1}(r+n')\\
&=(-1)^{m}\prod\limits_{r=0}^{m}(4(X-m)+2r+1)\prod\limits_{r=0}^{m-1}(r+n').
\end{split}&&
\end{flalign*}
We then have $S_m(n'+p,p)=S_{m,n'+p,p}$ and since $n=n'+p$, we have $n'=n-p$ and we obtain
\begin{flalign}\label{eq:factorized Smnp}
S_{m,n,p}=(-1)^{m}\prod\limits_{r=0}^{m}(4(p-m)+2r+1)\prod\limits_{r=0}^{m-1}(r+n-p).&&
\end{flalign}
By using Eq. (\ref{eq:factorized Smnp}) in Eq. (\ref{eq:v_2(p-m) with respect to Smnp}), we obtain
\begin{flalign*}
\begin{split}
v_{2(p-m),2n}^{[M,N,l]}&=\dfrac{(-1)^{m}l^{2(n+m-p)}(2n)!}{2^{n+m-p}(2(p-m))!(n-p+m)!m!}\prod\limits_{r=2(p-m)}^{p+n}\dfrac{1}{2r+1}\prod\limits_{r=0}^{m}(4(p-m)+2r+1)\prod\limits_{r=0}^{m-1}(r+n-p)\\
&=\dfrac{(-1)^{m}l^{2(n+m-p)}(2n)!}{2^{n+m-p}(2(p-m))!(n-p+m)!m!}\prod\limits_{r=2(p-m)}^{p+n}\dfrac{1}{2r+1}\prod\limits_{r=2(p-m)}^{2p-m}(2r+1)\prod\limits_{r=0}^{m-1}(r+n-p)\\
&=\dfrac{(-1)^{m}l^{2(n+m-p)}(2n)!}{2^{n+m-p}(2(p-m))!(n-p+m)!m!}\prod\limits_{r=2p-m+1}^{p+n}\dfrac{1}{2r+1}\prod\limits_{r=n-p}^{n-p+m-1}r.
\end{split}&&
\end{flalign*}
Hence, we finally conclude that
\begin{flalign*}
v_{2m,2n}^{[M,N,l]}&=\dfrac{(-1)^{p-m}l^{2(n-m)}(2n)!}{2^{n-m}(n-m)!(p-m)!(2m)!}\prod\limits_{r=n-p}^{n-m-1}r\prod\limits_{r=p+m+1}^{p+n}\dfrac{1}{2r+1}.&&
\end{flalign*}
The same steps can be followed to determine $v_{2m+1,2n+1}^{[M,N,l]}$.\indent$\blacksquare$
\end{proof}
Since the elements of $V^{[M,N,l]}$ are defined using products, recurrence relationships are more interesting for computation purposes.
\begin{corollary}\label{corollary}
Let $l\in\mathbb{R}_+^*$ and $M,N\in\mathbb{N}$ such that $M<N$. Let $p=\left\lfloor\dfrac{M}{2}\right\rfloor, p_1=\left\lfloor\dfrac{M-1}{2}\right\rfloor,q=\left\lfloor\dfrac{N}{2}\right\rfloor$ and $q_1=\left\lfloor\dfrac{N-1}{2}\right\rfloor$. We have
\begin{flalign*}
\begin{split}
&\forall (m,n)\in[\![0,p-1]\!]\times[\![p+1,q]\!],v_{2m+2,2n}^{[M,N,l]}=-\dfrac{(n-m)(p-m)(2(p+m)+3)}{l^2(n-m-1)(2m+1)(m+1)}v_{2m,2n}^{[M,N,l]},\\
&\forall (m,n)\in[\![0,p]\!]\times[\![p+1,q-1]\!],v_{2m,2n+2}^{[M,N,l]}=\dfrac{l^2(2n+1)(n+1)(n-m)}{(n-m+1)(n-p)(2(p+n)+3)}v_{2m,2n}^{[M,N,l]},\\
&\forall (m,n)\in[\![0,p_1-1]\!]\times[\![p_1+1,q_1]\!],v_{2m+3,2n+1}^{[M,N,l]}=-\dfrac{(n-m)(p-m)(2(p+m)+5)}{l^2(n-m-1)(2m+3)(m+1)}v_{2m+1,2n+1}^{[M,N,l]},\\
&\forall (m,n)\in[\![0,p_1]\!]\times[\![p_1+1,q_1-1]\!],v_{2m+1,2n+3}^{[M,N,l]}=\dfrac{l^2(2n+3)(n+1)(n-m)}{(n-m+1)(n-p)(2(p+n)+5)}v_{2m+1,2n+1}^{[M,N,l]}.
\end{split}&&
\end{flalign*}
\end{corollary}
\begin{proof}
All the formulas can be derived directly from Theorem \ref{theorem:V}.\indent$\blacksquare$
\end{proof}
Using the fact that $V^{[M,M,l]}=I_M$ and Theorem \ref{theorem:V}, we can finally determine the coefficients $(b_0,...,b_M)$ in the canonical basis for the polynomial degree reduction of $P$ in $\langle\cdot,\cdot\rangle_l$ as such:
\begin{flalign}\label{eq:(b_0,...,b_M)}
\begin{cases}
&\forall m\in[\![0,\left\lfloor\tfrac{M}{2}\right\rfloor]\!], b_{2m}=a_{2m}+\sum\limits_{n=\left\lfloor\tfrac{M}{2}\right\rfloor+1}^{\left\lfloor\tfrac{N}{2}\right\rfloor}v_{2m,2n}^{[M,N,l]}a_{2n}\\
&\forall m\in[\![0,\left\lfloor\tfrac{M-1}{2}\right\rfloor]\!],b_{2m+1}=a_{2m+1}+\sum\limits_{n=\left\lfloor\tfrac{M-1}{2}\right\rfloor+1}^{\left\lfloor\tfrac{N-1}{2}\right\rfloor}v_{2m+1,2n+1}^{[M,N,l]}a_{2n+1}\\
\end{cases}&&
\end{flalign}
\section{Computational complexity and examples}\label{section:computational}
In this section, we will compare the computational complexity of determining the canonical coefficients $(b_0,...,b_M)$ using Eq. (\ref{eq:(b_0,...,b_M)}), and the computational complexity of determining them by calculating $V^{[M,N,l]}$ as a matrix product. We will also present an example of how Eq. (\ref{eq:(b_0,...,b_M)}) can be used and an example showing its numerical stability compared to the classical approach.
\begin{proposition}\label{proposition:complexity direct}
The computational complexity of using Eq. (\ref{eq:(b_0,...,b_M)}) is $\mathcal{O}((N-M)M).$
\end{proposition}
\begin{proof}
According to Eq. (\ref{eq:(b_0,...,b_M)}), the number of summations and products that are required to determine one coefficient $b_m$ is at most $2\left(\left\lfloor\dfrac{N}{2}\right\rfloor-\left\lfloor\dfrac{M}{2}\right\rfloor\right)$. However, this is without counting how many operations are needed to determine $v_{m,n}^{[M,N,l]}.$ According to Corollary \ref{corollary}, the number of operations to determine $v_{m,n}^{[M,N,l]}$ is constant and does not depend on $M$ or $N$. Therefore, it can be ignored for complexity calculations. Since we need to calculate $M$ coefficients, the computational complexity is then of $\mathcal{O}\left(2M\left(\left\lfloor\dfrac{N}{2}\right\rfloor-\left\lfloor\dfrac{M}{2}\right\rfloor\right)\right)$. We have $2\left(\left\lfloor\dfrac{N}{2}\right\rfloor-\left\lfloor\dfrac{M}{2}\right\rfloor\right)\leq N-M+1$. Thus, we conclude that the computational complexity is $\mathcal{O}((N-M)M)$.\indent $\blacksquare$
\end{proof}
\begin{proposition}\label{proposition:complexity matrix}
The computational complexity of determining $(b_0,...,b_M)$ by calculating $V^{[M,N,l]}$ using a matrix product is $\mathcal{O}((N-M)M^2)$.
\end{proposition}
\begin{proof}
We can assume that to calculate $V^{[M,N,l]}$, we only need to determine the elements of $V^{[M,N,l]}$ whose columns are greater than $M$, since $V^{[M,M,l]}=I_M$. This is equivalent to performing a matrix product between a $(M+1)\times (M+1)$ matrix and a $(M+1)\times(N-M)$ matrix which has a $\mathcal{O}((N-M)M^2)$ complexity. The complexities of determining $T^{[M,N,l]}$ and $\left(T^{[M,M,l]}\right)^{-1}$ using Eqs. (\ref{eq:canonical to orthonormal}) and (\ref{eq:orthonormal to canonical}) are negligible compared to a matrix product. Therefore, the computational complexity of determining $(b_0,...,b_M)$ by calculating $V^{[M,N,l]}$ using a matrix product is $\mathcal{O}((N-M)M^2)$.\indent $\blacksquare$
\end{proof}
According to Proposition \ref{proposition:complexity direct} and Proposition \ref{proposition:complexity matrix}, directly computing $(b_0,...,b_M)$ using Eq. (\ref{eq:(b_0,...,b_M)}) is at least $M$ times less complex than calculating $V^{[M,N,l]}$. What follow are two examples showing the use of Theorem \ref{theorem:V} and Eq. (\ref{eq:(b_0,...,b_M)}).
\begin{example}$N=7,M=5$
\newline\newline
Let $l\in\mathbb{R}_{+}^*$. Using Theorem \ref{theorem:V}, we find that
\begin{flalign*}
V^{[5,7,l]}=
\begin{pmatrix}
1 & 0 & 0 & 0 & 0 & 0 &l^6\dfrac{1440}{66528} &0 \\ 
0 & 1 & 0 & 0 & 0 & 0 &0 & l^6\dfrac{10080}{123552}\\ 
0 & 0 & 1 & 0 & 0 & 0 &-l^4\dfrac{720}{1584} &0 \\ 
0 & 0 & 0 & 1 & 0 & 0 &0 &-l^4\dfrac{5040}{6864} \\ 
0 & 0 & 0 & 0 & 1 & 0 &l^2\dfrac{720}{528} &0\\ 
0 & 0 & 0 & 0 & 0 & 1 &0    &l^2\dfrac{5040}{3120}\\ 
\end{pmatrix}&&
\end{flalign*}
For a polynomial $P=\sum\limits_{k=0}^{7}a_kX^k$ where $(a_0,...,a_7)\in\mathbb{R}^8$, the canonical coefficients $(b_0,...,b_5)\in\mathbb{R}^6$ of the polynomial of degree $5$ that best approximates $P$ with respect to $\langle\cdot,\cdot\rangle_l$ are determined using Eq. (\ref{eq:(b_0,...,b_M)}) as such:
\begin{flalign*}
\begin{cases}
b_0&=a_0+l^6\dfrac{1440}{66528}a_6\vspace{0.5em}\\
b_1&=a_1+l^6\dfrac{10080}{123552}a_7\vspace{0.5em}\\
b_2&=a_2-l^4\dfrac{720}{1584}a_6\vspace{0.5em}\\
b_3&=a_3-l^4\dfrac{5040}{6864}a_7\vspace{0.5em}\\
b_4&=a_4+l^2\dfrac{720}{528}a_6\vspace{0.5em}\\
b_5&=a_5+l^2\dfrac{5040}{3120}a_7\\
\end{cases}&&
\end{flalign*}
Figure \ref{approximation example} shows an example of a polynomial of degree $7$ approximated by a polynomial of degree $5$ on the interval $[-5,5]$.
\begin{figure}[h!]
\centering
\includegraphics[scale=0.25]{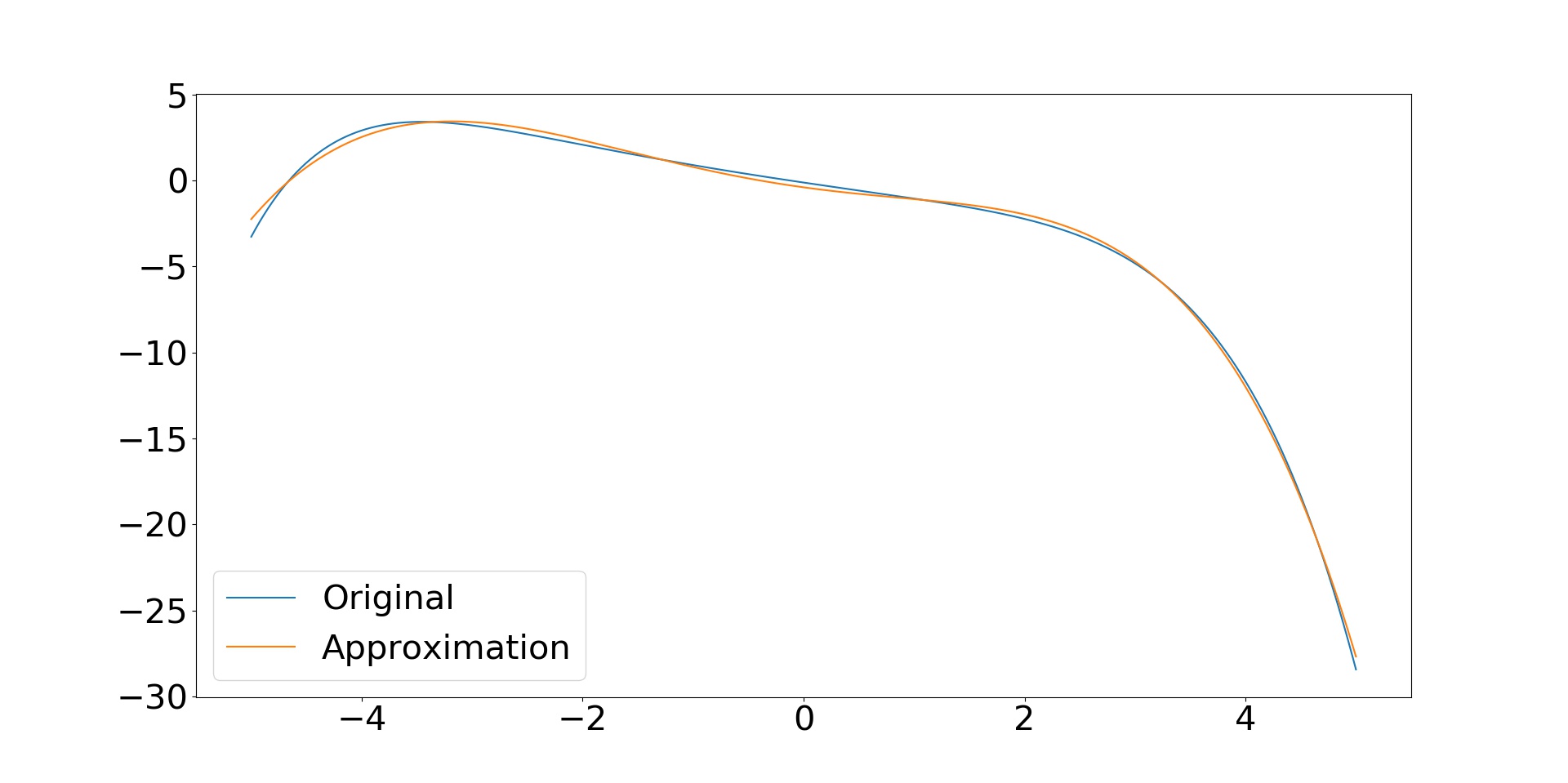}
\caption{A polynomial of degree $7$ (in blue) approximated by a polynomial of degree $5$ (in orange) on the interval $[-5,5]$.}
\label{approximation example}
\end{figure} 
\end{example}
\begin{example}$N=150,M=40,l=1$
\newline\newline
In this example, we want to reduce a polynomial of degree $150$ to a polynomial of degree $40$ on the interval $[-1,1]$. Figure \ref{fig:stability example}.(a) shows the original polynomial in blue and its approximation in orange using the direct formula while figure \ref{fig:stability example}.(b) shows the original polynomial in blue and its approximation in orange using the matrix multiplication approach. We notice that there are artifacts that corrupt the approximation using the matrix multiplication approach and this is mainly due to the floating point precision of the computation. Indeed, the matrix multiplication approach involves additions and multiplications between floating point quantities in order to determine one element of $V^{[M,N,l]}$. As a result, the direct formula is numerically more stable since, according to Theorem \ref{theorem:V}, every element of $V^{[M,N,l]}$ is the multiplication between a floating point quantity $l$ and a fraction.  In fact, a fraction is defined as a floating point quantity but it results from the division of two integer quantities. Consequently, the numerical floating point representation of the fraction will be better preserved than when using the matrix multiplication approach, hence the smooth approximation that is observed in Figure \ref{fig:stability example}.(a).
\begin{figure}[h!]
\centering
\subfigure[Polynomial reduction using the direct formula.]{\includegraphics[scale=0.4]{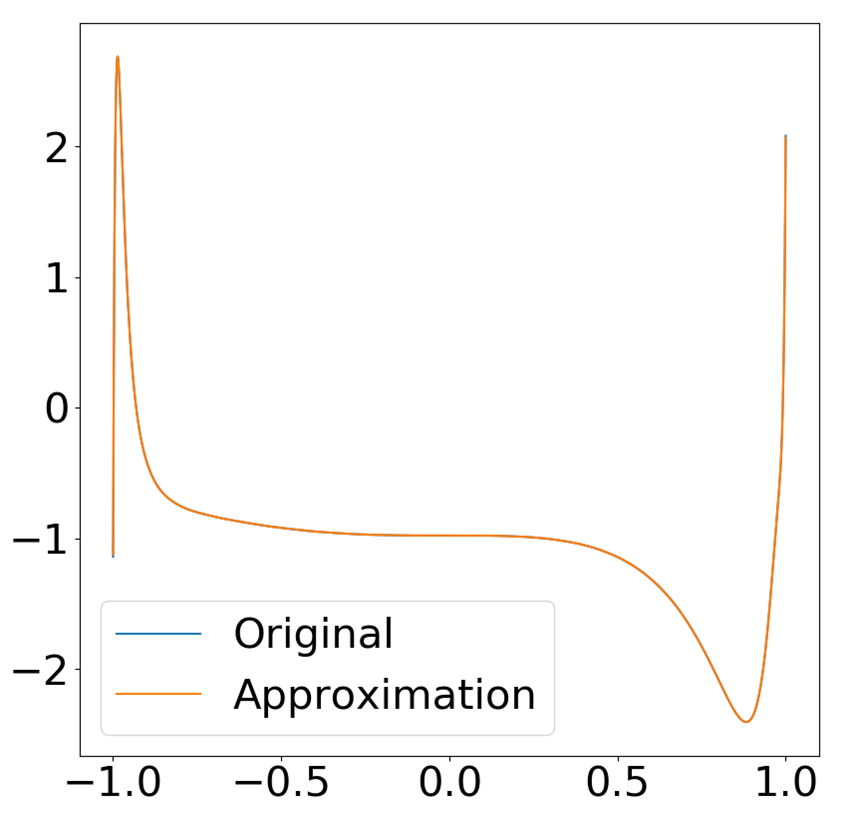}}\hspace{5mm}\subfigure[Polynomial reduction using the classical matrix multiplication approach.]{\includegraphics[scale=0.4]{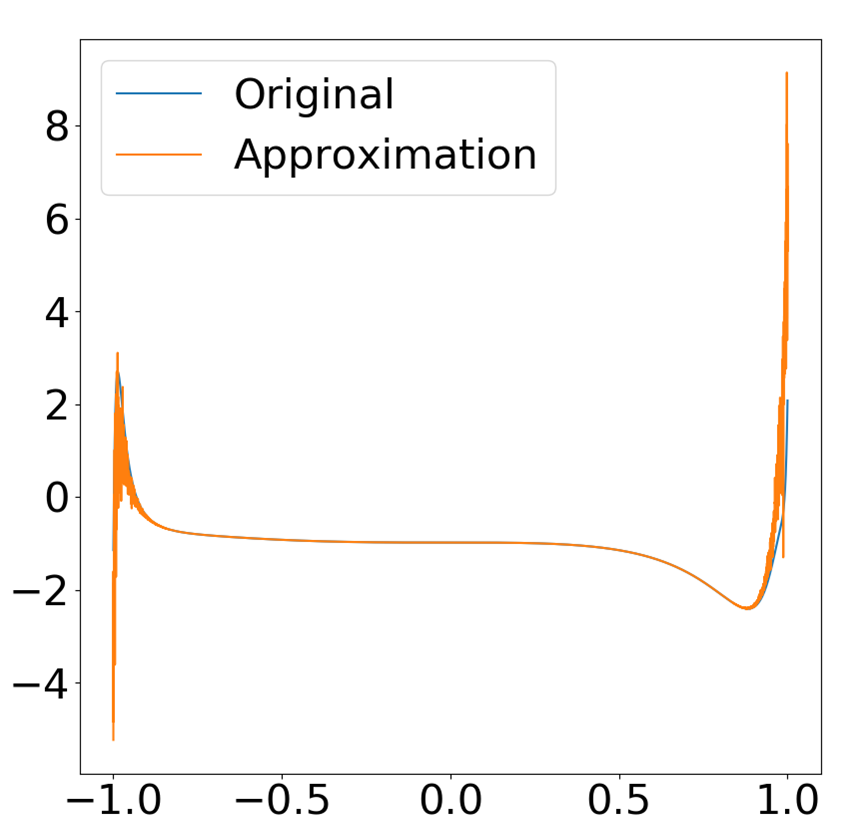}}
\caption{Comparison between the use of the direct formula and the use of the matrix multiplication approach on polynomial reduction.}
\label{fig:stability example}
\end{figure}
\end{example}
\section*{Acknowledgement}
This work was funded by the Mitacs Globalink Graduate Fellowship Program (no. FR41237) and the NSERC Discovery Grants Program (nos. 194376, 418413).
\section*{Declaration of interest}
None.
\bibliography{polynomial_reduction}
\bibliographystyle{ieeetr}
\end{document}